\documentclass[11pt,a4paper]{amsart}
\usepackage{amsmath,amssymb,amsthm,mathtools}
\usepackage[T1]{fontenc}
\usepackage{lmodern}
\usepackage{microtype}
\usepackage{hyperref}
\hypersetup{colorlinks=true,linkcolor=blue,citecolor=blue,urlcolor=blue}
\numberwithin{equation}{section}
\newtheorem{theorem}{Theorem}[section]
\newtheorem{lemma}[theorem]{Lemma}
\newtheorem{corollary}[theorem]{Corollary}

\newcommand{\ep}{\mathrm e}

\title[Weighted distribution of factorial valuations]{Weighted equidistribution of factorial prime valuations in fixed power bands}
\author{Ma Yicen}
\address{Zhejiang University, Hangzhou, China}
\date{September 27, 2026; research draft for independent review}

\begin{document}
\begin{abstract}
For a fixed modulus $h$ and residue class $c$, we study the logarithmic mass of primes $p$ in a fixed power band for which $\nu_p(n!)\equiv c\pmod h$. We give a proof of the asymptotic $\bigl((b-a)/h+o(1)\bigr)\log(n!)$ for $n^a<p\le n^b$, with $0\le a<b\le1$. The argument combines a fixed-order derivative estimate, Vaughan's identity, and finite Fourier approximation of a discontinuous valuation function. We derive a necessary capacity condition for infinite families of solutions to $P(x)=n!$. The proof and its relation to prior work require independent specialist review; the asymptotic provides no effective threshold and does not resolve the Brocard--Ramanujan equation.
\end{abstract}
\maketitle

\section{Introduction and statements}

Write $v_p(n!)=\sum_{j\ge1}\lfloor n/p^j\rfloor$ and $F_n=\log(n!)$. For a set of positive integers $A$ and an interval $J$ of primes, put
\begin{equation}\label{eq:mass}
 \Psi(n!,A;J)=\sum_{\substack{p\in J\\v_p(n!)\in A}}v_p(n!)\log p.
\end{equation}
The following fixed-parameter asymptotic is the central claim of this paper. In particular, the $o(1)$ is \emph{not} asserted to be uniform when the modulus or the endpoints vary with $n$.

\begin{theorem}[Weighted residue-class distribution]\label{thm:main}
Fix $h\ge2$, $c\in\mathbb Z/h\mathbb Z$, and $0\le a<b\le1$. Then, as $n\to\infty$ through all positive integers,
\begin{equation}\label{eq:main}
 \Psi\bigl(n!,c+h\mathbb Z;(n^a,n^b]\bigr)
 =\left(\frac{b-a}{h}+o(1)\right)F_n.
\end{equation}
\end{theorem}

For a finite union $A$ of arithmetic progressions, let $\delta(A)$ denote its natural density. Summing the disjoint residue classes in Theorem~\ref{thm:main} gives
\begin{corollary}\label{cor:BH}
For every fixed such $A$ and fixed $0\le a<b\le1$,
\begin{equation}\label{eq:BH}
 \Psi\bigl(n!,A;(n^a,n^b]\bigr)
 =\bigl((b-a)\delta(A)+o(1)\bigr)F_n.
\end{equation}
\end{corollary}
This is the qualitative formula proposed in Berend--Harmse \cite[Conjecture~5.1]{BH}, in their notation $a=1-\beta$ and $b=1-\beta_1$. An arithmetic progression restricted to positive integers can differ from its eventual residue classes in finitely many valuation values; Section~\ref{sec:application} checks that these exceptions have mass $O(n)$. We make no assertion here concerning priority or novelty relative to work beyond the cited sources.

The method works with fixed derivative orders depending on a power band. Its main analytic step, the prime exponential-sum estimate in Proposition~\ref{prop:prime}, is stated with the subinterval uniformity needed in both the Type II estimate and partial summation.

\section{A derivative estimate on every subinterval}\label{sec:derivative}

Put $\ep(t)=\exp(2\pi i t)$. We use the $d$th derivative estimate of Arias de Reyna \cite{AdR}. If $d\ge3$, $K=2^d$, $\lfloor Y\rfloor>d$, and $0<\lambda\le g^{(d)}(t)\le\Lambda$ on an interval of length $Y$, then
\begin{equation}\label{eq:vdc}
 \frac1Y\left|\sum_{X<m\le X+Y}\ep(g(m))\right|
 \le 11\max\left\{\left(\frac{\Lambda}{\lambda Y}\right)^{2/K},
 \left(\frac{\Lambda^2}{\lambda}\right)^{1/(K-2)},
 (\lambda Y^d)^{-2/K}\right\}.
\end{equation}
Replacing $g$ by $-g$ handles a negative derivative. Only fixed $d$ and implicit constants depending on $d$ will be used.

\begin{lemma}[Subinterval form]\label{lem:subinterval}
Fix $u,c_0,C_0>0$. Suppose $P^u\ll M\ll P$, $P^{c_0}\ll T\ll P^{C_0}$, and $g$ is real and smooth on a container of length $O(M)$ contained in $[M,CM]$ for a fixed $C$. Suppose, for a fixed integer $d\ge3$ with $du>C_0$, that $g^{(d)}$ has one sign throughout the container and
\[
 |g^{(d)}(t)|\asymp TM^{-d}
\]
there, with fixed comparison constants. There exists $\delta>0$ such that on \emph{every} subinterval $I$ of that container,
\begin{equation}\label{eq:subinterval}
 \left|\sum_{m\in I\cap\mathbb Z}\ep(g(m))\right|\ll MP^{-\delta}.
\end{equation}
The constants may depend on the fixed data, including $d$.
\end{lemma}
\begin{proof}
For an interval of length $Y>d+1$, apply \eqref{eq:vdc} with $\lambda,\Lambda\asymp TM^{-d}$. Multiplying by $Y$ bounds the three terms, up to fixed constants, by
\[
 M^{2/K}Y^{1-2/K},\qquad
 Y(TM^{-d})^{1/(K-2)},\qquad
 (TM^{-d})^{-2/K}Y^{1-2d/K}.
\]
All three exponents of $Y$ are positive for $d\ge3$, so $Y\ll M$ yields
\[
 \ll M\max\bigl\{M^{-2/K},(TM^{-d})^{1/(K-2)},T^{-2/K}\bigr\}.
\]
The hypotheses $M\gg P^u$, $T\gg P^{c_0}$, and $T M^{-d}\ll P^{C_0-du}$ give a fixed negative power of $P$ in each term. For $Y\le d+1$, use the trivial $O_d(1)$ bound and decrease $\delta$ if necessary. Endpoint conventions change the sum by at most $O(1)$.
\end{proof}

\section{Reciprocal phases on primes}\label{sec:primes}

Fix $r\ge1$ and real numbers
\begin{equation}\label{eq:band}
 \frac1{r+1}<a<b<\frac1r.
\end{equation}
Fix $h\ge2$ and a nonzero $\mathbf t=(t_1,\ldots,t_r)\in\mathbb Z^r$. Throughout this section, all of them are held fixed. Set
\begin{equation}\label{eq:phase}
 f(x)=\frac nh\sum_{j=1}^r t_jx^{-j},\quad
 j_0=\min\{j:t_j\ne0\},\quad T=nP^{-j_0}.
\end{equation}
For a dyadic container $(P,2P]$ meeting $(n^a,n^b]$, one has $P^\kappa\ll n\ll P^{1/\kappa}$ for a suitable fixed $\kappa>0$, and consequently, for fixed $c_0,C_0>0$,
\begin{equation}\label{eq:T-range}
 P^{c_0}\ll T\ll P^{C_0}.
\end{equation}
For instance, one can take $c_0<1/b-r$ and $C_0>1/a$, after harmless changes of constants. Note that $a>0$ in this section.

\begin{theorem}[Prime-phase bound]\label{prop:prime}
For some $\eta>0$ depending on the fixed parameters, every subinterval $I$ of $(P,2P]\cap(n^a,n^b]$ satisfies
\begin{equation}\label{eq:prime-bound}
 \sum_{p\in I}(\log p)\ep(f(p))\ll P^{1-\eta}.
\end{equation}
The estimate also holds with the prime sum replaced by $\sum_{m\in I}\Lambda(m)\ep(f(m))$.
\end{theorem}
\begin{proof}
We apply Vaughan's identity with $U=V=P^{1/4}$, ignoring harmless integer parts:
\begin{equation}\label{eq:vaughan}
 \Lambda=\Lambda_{\le V}+\mu_{\le U}*\log
 -\mu_{\le U}*1*\Lambda_{\le V}
 +\mu_{>U}*1*\Lambda_{>V}.
\end{equation}
Here the subscripts denote support conditions on the argument. The first term vanishes on $(P,2P]$ for large $P$. The middle terms yield Type I sums with outer variable $q\le UV=P^{1/2}$ and inner variable $m\asymp M=P/q\gg P^{1/2}$. The coefficient of $q$ in the third term is bounded by
\[
 \left|\sum_{\substack{ab=q\\a\le U,\,b\le V}}\mu(a)\Lambda(b)\right|
 \le\sum_{b\mid q}\Lambda(b)=\log q.
\]
The inner weights in the second term are $\log m$.

For any fixed derivative order $d$, differentiating $f(qm)$ shows that its $j_0$-term has magnitude $\asymp TM^{-d}$ on the entire inner container. For $j>j_0$ the ratio to that term is $O(P^{-(j-j_0)})$, uniformly in $q$, because $qm\asymp P$. Thus, for large $P$, the derivative is of constant sign and is $\asymp TM^{-d}$. Choose $d\ge3$ with $d/2>C_0$ and apply Lemma~\ref{lem:subinterval}. Partial summation removes $\log m$ with an extra factor $O(\log P)$; the $q$-coefficients cost another $O(\log P)$. Summing $M\asymp P/q$ over $q\le P^{1/2}$ costs $\sum q^{-1}\ll\log P$. The Type I contribution is $\ll P^{1-\delta}(\log P)^3$.

For the last term of \eqref{eq:vaughan}, put
\[
 b_\ell=\sum_{\substack{v\mid\ell\\v>V}}\Lambda(v),\qquad 0\le b_\ell\le\log\ell.
\]
Partition $m>U$, $\ell>V$ into $O((\log P)^2)$ dyadic blocks $m\asymp M$, $\ell\asymp L$, where $M,L\gg P^{1/4}$ and $ML\asymp P$. The condition $m\ell\in I$ gives a consecutive interval of $\ell$ for each $m$; when the square is expanded, the set of admissible $m$ for a pair $(\ell_1,\ell_2)$ is the intersection of two intervals, hence again an interval. By Cauchy--Schwarz and $|\mu(m)|\le1$, the squared magnitude of one block is at most
\begin{equation}\label{eq:type2-cauchy}
 \ll M(\log P)^2\sum_{\ell_1,\ell_2\asymp L}
 \left|\sum_{m\in I(\ell_1,\ell_2)}
 \ep(f(m\ell_1)-f(m\ell_2))\right|.
\end{equation}
Truncation to a dyadic $m$-block is included in $I(\ell_1,\ell_2)$.

Suppose $\ell_1\ne\ell_2$ and write $\Delta=|\ell_1-\ell_2|$. The $d$th derivative in $m$ of the difference phase has magnitude
\begin{equation}\label{eq:diff-derivative}
 \asymp T\frac{\Delta}{L}M^{-d}
\end{equation}
and one sign, uniformly even for $\Delta=1$. Indeed, for $j=j_0+s$,
\begin{equation}\label{eq:ratio}
 \frac{|\ell_1^{-j}-\ell_2^{-j}|}
 {|\ell_1^{-j_0}-\ell_2^{-j_0}|}
 \le\frac{j}{j_0}\min(\ell_1,\ell_2)^{-s},
\end{equation}
as follows by integrating the two derivatives between $\ell_1$ and $\ell_2$. Multiplication by $m^{-s}$ and the fixed ratio of rising factorials bounds the higher-$j$ contribution relative to $j_0$ by $O(P^{-s})$. Since $m\ell_i\asymp P$, this estimate is uniform in $\Delta$.

The diagonal and the pairs with $\Delta\le LP^{-c_0/2}$ number
$O(L^2P^{-c_0/2}+L)$ and contribute at most $O(M)$ each. For every remaining pair, $T\Delta/L\gg P^{c_0/2}$ and $T\Delta/L\ll P^{C_0}$. Apply Lemma~\ref{lem:subinterval} with $u=1/4$, $c_0$ replaced by $c_0/2$, and fixed $d>4C_0$ to obtain an $O(MP^{-\delta})$ bound for the inner sum. Inserting these estimates in \eqref{eq:type2-cauchy} yields
\begin{equation}\label{eq:type2-bound}
 |\text{block}|^2\ll M^2L^2(\log P)^2
 \bigl(P^{-c_0/2}+L^{-1}+P^{-\delta}\bigr).
\end{equation}
Thus Type II also has a fixed power saving after summing its blocks. The same fixed $d>4C_0$ could have been used for Type I. Finally, prime powers in $(P,2P]$ contribute $O(P^{1/2}\log^2P)$ by a crude estimate; removing them completes the proof.
\end{proof}

\section{The valuation step function}\label{sec:fourier}

In the band \eqref{eq:band}, $p^{r+1}>n$ and $p^r<n$ for large $n$, so
\begin{equation}\label{eq:exactvaluation}
 v_p(n!)=\sum_{j=1}^r\left\lfloor\frac n{p^j}\right\rfloor.
\end{equation}
Fix $1\le z<h$ and define the one-periodic step function
\[
 g_z(t)=\ep\left(\frac zh\lfloor ht\rfloor\right),\qquad
 G_z(t_1,\ldots,t_r)=\prod_{j=1}^r g_z(t_j).
\]
It has absolute value one away from the immaterial jump convention, and its torus integral is zero because $\int_0^1g_z(t)\,dt=h^{-1}\sum_{k=0}^{h-1}\ep(zk/h)=0$. Moreover,
\begin{equation}\label{eq:G-relation}
 \ep\bigl(zv_p(n!)/h\bigr)
 =G_z\left(\frac n{hp},\frac n{hp^2},\ldots,\frac n{hp^r}\right).
\end{equation}

We spell out the approximation at the jumps. For each fixed $\varepsilon>0$, there exist finite trigonometric polynomials $Q_\varepsilon,H_\varepsilon$ on the $r$-torus such that $H_\varepsilon\ge0$, $|G_z-Q_\varepsilon|\le H_\varepsilon$ pointwise (including the jumps), and $\int H_\varepsilon\le\varepsilon$. To see this, take neighborhoods of all coordinate jump hyperplanes of total measure as small as desired. Approximate $G_z$ uniformly by a trigonometric polynomial off these neighborhoods, then majorize the absolute error by a continuous nonnegative function that is small off them and bounded on them. Uniform approximation of this continuous majorant by a polynomial, followed by adding a small constant, gives $H_\varepsilon$. Similarly $\int Q_\varepsilon=O(\varepsilon)$ since $\int G_z=0$. Every degree in this construction depends only on the fixed $\varepsilon,h,r$.

For every nonconstant Fourier monomial, Theorem~\ref{prop:prime} applies to its frequency vector; its constant term is treated by $\sum_{p\le2P}\log p\ll P$. It follows uniformly over subintervals $I$ of a relevant dyadic container that
\begin{equation}\label{eq:eps-dyadic}
 \left|\sum_{p\in I}(\log p)\ep\bigl(zv_p(n!)/h\bigr)\right|
 \ll\varepsilon P+O_\varepsilon(P^{1-\eta_\varepsilon}).
\end{equation}
Here $\eta_\varepsilon>0$ exists because only finitely many frequencies occur.

Partial summation, using the bound for \emph{all} initial subintervals, gives on each dyadic piece
\begin{equation}\label{eq:weighted-dyadic}
 \sum_{p\in(P,2P]\cap(n^a,n^b]}
 \frac{\log p}{p}\ep\bigl(zv_p(n!)/h\bigr)
 =O(\varepsilon)+O_\varepsilon(P^{-\eta_\varepsilon}).
\end{equation}
There are $O(\log n)$ pieces, each with $P\gg n^a$, so after division by $\log n$ the second terms tend to zero for each fixed $\varepsilon$. Then let $\varepsilon\downarrow0$. We have proved
\begin{lemma}\label{lem:weighted}
Under \eqref{eq:band}, for fixed $h$ and $1\le z<h$,
\begin{equation}\label{eq:weighted}
 \sum_{n^a<p\le n^b}\frac{\log p}{p}
 \ep\bigl(zv_p(n!)/h\bigr)=o(\log n).
\end{equation}
\end{lemma}

\section{Completion of the main theorem}\label{sec:completion}

Legendre's formula gives, uniformly for $p\le n$,
\begin{equation}\label{eq:legendre}
 v_p(n!)=\frac np+O\left(1+\frac n{p(p-1)}\right).
\end{equation}
Since $\vartheta(n)=\sum_{p\le n}\log p=O(n)$ and
$\sum_p(\log p)/[p(p-1)]<\infty$, the total error in replacing $v_p(n!)\log p$ by $(n/p)\log p$ over \emph{any} prime set below $n$ is $O(n)$. Lemma~\ref{lem:weighted} therefore implies, for $z\not\equiv0\pmod h$ and \eqref{eq:band},
\begin{equation}\label{eq:character}
 \sum_{n^a<p\le n^b}v_p(n!)\log p\,
 \ep\bigl(zv_p(n!)/h\bigr)=o(n\log n)=o(F_n).
\end{equation}
The prime Mertens estimate $\sum_{p\le x}(\log p)/p=\log x+O(1)$ gives, for any fixed $0\le a<b\le1$,
\begin{equation}\label{eq:total}
 \sum_{n^a<p\le n^b}v_p(n!)\log p
 =(b-a)n\log n+O(n)=(b-a)F_n+o(F_n).
\end{equation}
Fourier inversion in $\mathbb Z/h\mathbb Z$ proves Theorem~\ref{thm:main} whenever \eqref{eq:band} holds.

To remove the restriction \eqref{eq:band}, fix $\varepsilon>0$. Delete the end regions with logarithmic coordinates $[0,\varepsilon]$ and $[1-\varepsilon,1]$, and fixed neighborhoods of the finitely many points $1/j$ lying in $[\varepsilon,1-\varepsilon]$, with total logarithmic width at most $\varepsilon$. The remaining part of $[a,b]$ is a finite union of closed subbands strictly inside intervals $(1/(r+1),1/r)$. Their endpoints and $r$ are fixed independently of $n$, so the proved result applies on each. By \eqref{eq:total}, the total $v_p(n!)\log p$ mass in the deleted bands is $O(\varepsilon F_n)+O_\varepsilon(n)$; the same upper bound holds for any residue class. Sum the preserved subbands, first let $n\to\infty$, and then let $\varepsilon\downarrow0$. This proves \eqref{eq:main}, including $a=0$ and $b=1$.

For completeness, if two sets $A,A'$ differ in a fixed finite set of positive valuation values $\{1,\ldots,K\}$, the corresponding mass difference is at most
\[
 \sum_{p:v_p(n!)\le K}v_p(n!)\log p\le K\vartheta(n)=O_K(n).
\]
This proves the endpoint qualification preceding Corollary~\ref{cor:BH}.

\section{A capacity condition for polynomial--factorial equations}\label{sec:application}

Suppose
\begin{equation}\label{eq:factorization}
 P(X)=c\prod_{i=1}^s f_i(X)^{m_i}\in\mathbb Z[X],
\end{equation}
where $c\in\mathbb Q^\times$, the $f_i\in\mathbb Z[X]$ are nonconstant and pairwise coprime over $\mathbb Q[X]$, and $m_i\ge1$. Put $d_i=\deg f_i$, $d=\sum_i m_id_i$, and $L=\operatorname{lcm}_i m_i$. For $I\subseteq\{1,\ldots,s\}$ define
\[
 \mathcal R(I)=\{v\bmod L:\ m_i\mid v\text{ for at least one }i\in I\}.
\]

\begin{theorem}[Necessary capacity inequalities]\label{thm:capacity}
If $P(x)=n!$ has infinitely many integral solutions $(x,n)$ with $n\ge1$, then for every subset $I\subseteq\{1,\ldots,s\}$,
\begin{equation}\label{eq:capacity}
 \frac{\sum_{i\in I}m_id_i}{d}\le\frac{|\mathcal R(I)|}{L}.
\end{equation}
Thus a strict violation of any inequality implies finiteness of integral solutions.
\end{theorem}
\begin{proof}
An unbounded family of solutions has $n\to\infty$ and $|x|\to\infty$: for fixed $n$ the polynomial $P(X)-n!$ has finitely many integer roots, and a bounded $x$ gives only bounded $n$. Taking logarithms of absolute values in \eqref{eq:factorization} yields
\begin{equation}\label{eq:degree-mass}
 \sum_{i\in I}m_i\log|f_i(x)|
 =\frac{\sum_{i\in I}m_id_i}{d}F_n+O_P(1),
\end{equation}
since $\log|f_i(x)|=d_i\log|x|+O_P(1)$ for $|x|\to\infty$ and $F_n=d\log|x|+O_P(1)$.

Only finitely many primes divide a fixed numerator or denominator of $c$, or a fixed nonzero integer common multiple of all pairwise resultants $\operatorname{Res}(f_i,f_j)$. Outside this set, the integer values $f_i(x)$ cannot share a prime divisor. For any such prime that contributes to the left side of \eqref{eq:degree-mass}, it belongs to a unique $f_i(x)$ and $v_p(n!)=m_i v_p(f_i(x))$, so its valuation class lies in $\mathcal R(I)$. At any one excluded prime, $v_p(n!)\log p=O_P(n)$, by Legendre's formula. Consequently \eqref{eq:degree-mass} is at most
\[
 \sum_{\substack{p\le n\\v_p(n!)\bmod L\in\mathcal R(I)}}
 v_p(n!)\log p+O_P(n)
 =\left(\frac{|\mathcal R(I)|}{L}+o(1)\right)F_n,
\]
where the last equality is Theorem~\ref{thm:main} with $(a,b)=(0,1)$. Divide by $F_n$ to obtain \eqref{eq:capacity}.
\end{proof}

\begin{corollary}\label{cor:example}
For every fixed integer $r\ge2$, the equation $x^r(x+1)=n!$ has only finitely many integral solutions $(x,n)$.
\end{corollary}
\begin{proof}
For the singleton factor $X^r$, \eqref{eq:capacity} would require $r/(r+1)\le1/r$, whereas $r^2>r+1$ for $r\ge2$.
\end{proof}

\section*{Status of this draft}
This manuscript formalizes the argument in the author's research note of September 27, 2026. The fixed-parameter prime-phase estimate, in particular the Type II subinterval argument and the passage across the discontinuities of $G_z$, merits independent specialist verification. No explicit numerical onset for the asymptotic is claimed.

\end{document}